\documentclass{amsart}[11pt]
\usepackage{color,amssymb,comment}
\usepackage{graphicx}
\usepackage{amsrefs}
\usepackage{enumerate}

\newcommand{\R}{\mathbb{R}}

\newcommand{\infinity}{\infty}

\newcommand{\bdy}{\partial}

\newcommand{\e}{\operatorname{e}}

\theoremstyle{plain}
\newtheorem{lemma}{Lemma}  
\newtheorem{theorem}[lemma]{Theorem}

\newtheorem{proposition}[lemma]{Proposition}
\newtheorem{definition}[lemma]{Definition}

\renewcommand{\v}{\mathbf{v}}
\newcommand{\Laplace}{\triangle}
\newcommand{\vol}{\operatorname{vol}}

\begin{document}

\author{Samuel Lisi}
\title{Dividing sets as nodal sets of eigenfunctions of a Laplacian}
\date{31 March 2010}
\begin{abstract}
We show that for any convex surface $S$ in a contact 3-manifold,
there exists a metric on $S$ and a neighbourhood contact isotopic to
$S \times I$ with the contact structure given by $\ker ( u dt - \star du)$
where $u$ is an eigenfunction of the Laplacian on $S$ and $\star$ is the
Hodge star from the metric on $S$.
This answers a question posed by Komendarczyk \cite{Komendarczyk}.
\end{abstract}

\maketitle
Given a convex surface $S$ in a contact 3-manifold, we show the
existence of a metric defined on a tubular neighbourhood of $S$, adapted
to the contact structure, for which the dividing curves are nodal curves
of an eigenfunction of the Laplacian on $S$.  In our construction, we show
that any dividing set may be realized in this way.
This addresses two questions raised by Komendarczyk \cite{Komendarczyk}.

\begin{definition}
  Let $(M, \xi)$ be a contact 3-manifold with co-oriented contact
  structure
  $\xi$.  A metric $g$ on $M$ is \textit{adapted} to the contact structure
  if there exists a contact form $\alpha$ generating the contact structure
  so that $\star \alpha = d \alpha$.
\end{definition}

A class of examples of metrics adapted to a contact structure $\xi =
\ker \alpha$
is given by taking an almost complex structure $J$ on $\xi$ compatible
with $d\alpha$, i.e.~so that $d\alpha(\cdot, J \cdot)$ is a metric
on $\xi$.
Then, we construct a metric adapted to $\xi$ by taking
$g = \alpha^2 + d\alpha(\cdot, J\cdot)$.

\begin{definition}[Convex surface \cite{GirouxConvex}]
  A surface $S$ in a contact 3-manifold $(M, \xi)$ is \textit{convex}
  if there exists a (local) vector field $\v$ transverse to $S$,
  so that $L_\v \xi = 0$.\\
  The \textit{dividing set} is the set of all points on $S$
  where $\v \in \xi$. (The contact condition forces this to be an embedded
  multicurve in $S$.)\\
  The dividing set divides the surface $S$ into two open submanifolds,
  $S_+$ on which $\v$ is positively transverse to $\xi$ and $S_-$ on which
  $\v$ is negatively transverse to $\xi$. 
\end{definition}

Our main result can then be stated as follows:
\begin{theorem} \label{T:metric}
  Let $S$ be a convex surface in the contact 3-manifold $(M, \xi)$.
  Then, there exist an isotopic surface $S'$, an adapted metric $g$ and
  an eigenfunction $u$ of $\Laplace_g|_{S'}$ so that a neighbourhood
  of $S'$ is contactomorphic to $S' \times I$ with the contact structure
  \[ \ker ( u dt - \star du). \]
  Furthermore, this metric $g$ may be taken to be $t$--translation
  invariant.
\end{theorem}
(Here, we take the conventions $\alpha \wedge \star \alpha = |\alpha|_g^2 d \vol$, and 
$\Laplace_g u = \star d \star d u$.)

Komendarczyk \cite{Komendarczyk} proved this result in the special
case that the dividing set has one connected component, by using
techniques from spectral geometry.  In recent personal communication,
he has explained to the author a possible extension of these methods
to the general case.  In contrast to his methods, we prove the
result using ``soft'' techniques in contact topology.

The relationship between adapted metrics and contact topology has just
recently begun to be exploited, notably by Etnyre and Komendarczyk
\cite{EtnyreKomen} and by Etnyre, Komendarczyk and Massot in a contact version
of the sphere $1/4$ pinching result~\cite{EtnyreKomenMassot}.

The main interest of convex surfaces in contact topology comes from
Giroux's flexibility theorem\cite{GirouxConvex}:
\begin{theorem}
  Suppose $\Sigma$ is a closed convex surface in $(M, \xi)$, with
  transverse
  contact vector field $\v$ and dividing curves $\Gamma$.
  Suppose $\mathcal F$ is a singular foliation on
  $\Sigma$ divided by $\Gamma$.
  Then, there exists an isotopy $\phi_s$, $s\in [0, 1]$, so that
  $\phi_0(\Sigma) = \Sigma$, $\xi|_{\phi_1(\Sigma)} = \phi_1(\mathcal F)$,
  $\phi$ fixed on $\Gamma$ and $\phi_s(\Sigma)$ transverse to $\v$
  for all $s$.
\end{theorem}
Heuristically speaking, this tells us that the neighbourhood of $S$
is described, up to isotopy, by the dividing curves.  
In particular, if $\Sigma \times \R$
admits two translation invariant contact structures giving the same
dividing curves on $S$ and cutting out the same $S_\pm$ regions,
then the two contact structures are isotopic.

The proof of Theorem \ref{T:metric} will be by constructing the metric
and the eigenfunction, and will exploit the soft aspects of symplectic
and contact topology.  Instead of constructing these directly, we will
construct a symplectic form on $S$ and an almost complex structure
compatible with it.  Reformulated in this way, we obtain:
\begin{theorem} \label{T:contact}
  Let $S$ be a closed, connected surface, and consider the $t$-invariant
  contact structure on $\R \times S$ given by $\xi_0 = \ker(\alpha_0 =
  fdt + \beta)$.
  Denote the dividing curves by $\Gamma = f^{-1}(0) \subset S$.

  Orient $S$ by $i_{\partial_t}(\alpha_0 \wedge d\alpha_0)$.
  Then, there exist an area form $\Omega$ on $S$, compatible with the
  orientation, a compatible complex structure $j$, and a function $u :
  S \rightarrow \R$ with
  \[
  d( du \circ j ) = u \Omega \text{ and }
  u^2 + |du|^2 > 0,
  \]
  and so that $u^{-1}(0) = f^{-1}(0)$, and so the contact form 
  $u dt + du \circ j$ induces the same $S_+$ and $S_-$ regions.
\end{theorem}

\begin{proof}[Proof of Theorem \ref{T:metric}]
  We will now show that Theorem \ref{T:contact} implies \ref{T:metric}.

Let $S$ be a convex surface in $(M, \xi)$ with transverse contact
vector field
$\v$.  Let $\Gamma$ be the dividing set.  Then, by following the flow
of $\v$,
there exists a neighbourhood of $S$ in $M$ contactomorphic to a
neighbourhood of
$\{0\} \times S$ in $\R \times S$, with contact structure given by the
contact form
\[
\alpha_0 = fdt + \beta,
\]
where $f$ and $\beta$ are a function and a one-form on $S$ respectively.
Then, $f^{-1}( 0) = \Gamma$.

Let $u$, $j$ and $\Omega$ be as Theorem \ref{T:contact}.
This then gives a contact form on $\R \times S$ by $\alpha_1 = u dt +
du \circ j$.
Define a metric on $\R \times S$ by setting $g = dt^2 + \Omega(\cdot,
j \cdot)$.
We now observe $\Laplace_g|_{S} u \Omega = -d (du \circ j )$, so $u$
as in Theorem
\ref{T:contact} indeed is an eigenfunction of the Laplace operator of $g$
restricted to $S$.
It now remains to verify that this metric $g$
is adapted to the contact structure.  We observe that
$d \vol_g = dt \wedge \Omega$ and that $g( \partial_t, \cdot) = dt$,
$g(-X_u, \cdot) = du\circ j$ and $g(-jX_u, \cdot) = du$.  It follows that
$\star dt = \Omega$ and $\star du\circ j = -dt \wedge du$.  Thus,
\[
\star (udt + du\circ j) = u \Omega - dt \wedge du = d \alpha_1.
\]
as required.

We now have two translation invariant contact structures on $\R \times
S$, generated
by the contact forms $\alpha_0$ and $\alpha_1$.  The dividing sets and
induced orientations
on $S$ are the same, so by Giroux's Flexibility Theorem,
the contact structures are contact isotopic.
Following the image of the isotopy in $M$,
and restricting to a sufficiently small interval around $t = 0$ gives
the resulting $S'$.
\end{proof}

The main result, Theorem \ref{T:contact}, is a corollary of the following
result:
\begin{proposition} \label{P:main}
Let $S$ be a closed, connected surface and
$\Gamma \subset S$ be a collection of embedded circles dividing
$S$ into two regions, so that :
\[
S \cong S_- \cup [-1, 1] \times \Gamma \cup S_+,
\]
where $S_\pm$ are two open submanifolds of $S$.

Then, there exist a smooth function $u : S \to \R$, an area form
$\Omega$ and a compatible complex structure $j$ on $S$
so that :
\begin{align*}
d(du \circ j ) &= u \Omega \\
u^{-1}(0) &= \Gamma \\
u^2 + |du|^2 &> 0.
\end{align*}
\end{proposition}

The remainder of this paper is a proof of Proposition \ref{P:main}.

\section*{Proof of Proposition \ref{P:main}}

The key step in the proof of Proposition \ref{P:main} is the following
Lemma, whose proof will come later. 
\begin{lemma} \label{L:main}
  Let $S = S_- \cup [-1, 1] \times \Gamma \cup S_+$ as in the hypothesis of \ref{P:main}.

There exist an area form $\omega$, a compatible complex structure $j$,
and a real valued function $F$ on $S$ so that
for some constants $C > 0$ and $\epsilon > 0$ the following properties hold:
\begin{enumerate}[(i)]
  \item $\max_S |F | < \frac{\pi}{2}$, $F^{-1}(0) = \{ 0 \} \times \Gamma
    \subset [-1,1] \times \Gamma$, 
  \item $F < 0 \text{ on } S_- \cup [-1,0) \times \Gamma$,
	$F > 0 \text{ on } (0, 1] \times \Gamma \cup S_+$,
	$dF \ne 0 \text{ on } [-1, 1]\times \Gamma.$
  \item $d( dF \circ j ) \le 0$ on $[-1, 0] \times \Gamma$,  and 
		$d( dF \circ j ) < 0 \text{ on } S_-$
 \item $d( dF \circ j) \ge 0 \text{ on } [0, 1] \times \Gamma$ 
	and $d( dF \circ j) > 0$ on $S_+$
 \item for $(s, t) \in (-\epsilon, \epsilon) \times \Gamma$,
    $F(s, t) = Cs$ and  $\omega = ds \wedge dt$.
 \end{enumerate}
\end{lemma}

\begin{proof}[Proof of Proposition \ref{P:main}]

Let $F$, $j$ and $\omega$ be as in Lemma \ref{L:main}.
Define a real valued function on $S$ by
\[
u = \sin (F).
\]

Define $\triangle F$ by $-(\triangle F) \omega = d( dF \circ j )$.
Then, we obtain :
\begin{equation*}
\begin{split}
d( d u \circ j ) &= u | dF|^2 \omega + \cos (F) d( dF \circ j) \\
	&= u \left ( |dF|^2 - \frac{\cos(F)}{\sin(F)} \triangle F \right)
	\omega.
\end{split}
\end{equation*}

Observe that from the definition of $F$, $\triangle F = 0$ on $(-\epsilon, \epsilon) \times \Gamma$.
Since $|F| < \pi/2$, $\sin(F)$ has the same sign as $F$.  
Thus, $- \frac{1}{\sin(F)} \triangle F \ge 0$ is nonsingular, and only vanishes
in a subset of $[-1,1]\times \Gamma$, where $dF$ is non-vanishing.
Hence, 
\[
|dF|^2 - \frac{\cos(F)}{\sin(F)} \triangle F > 0.
\]
Thus, by taking
\[
\Omega =\left ( |dF|^2 - \frac{\cos(F)}{\sin(F)} \triangle F \right
) \omega
\]
we obtain a volume form on $S$ so that $d( du \circ j ) = u \Omega$.
We claim this triple of $u$, $\omega$ and $j$ has the desired properties.

Since $| F | < \frac{\pi}{2}$, it follows that
$u > 0$ on $(0, 1] \times \Gamma \cup S_+$, and that
$u < 0$ on $S_- \cup [-1, 0) \times \Gamma$.  Furthermore,
$u^{-1}(0) = F^{-1}(0) = \{ 0 \} \times \Gamma$.

To show $u^2 + |du|^2 > 0$, it suffices to check near
$u^{-1}(0) = \{0 \} \times \Gamma$.
Note, however, that in a neighbourhood of $\{0\} \times \Gamma$, $F(s,
t) = Cs$,
for some positive constant $C$,
and thus $du = \cos(F) dF = \cos(Cs) C ds$, which is non-vanishing
in a neighbourhood of $\{ 0 \} \times \Gamma$.	This completes the proof
of Proposition \ref{P:main}.
\end{proof}

We now prove the key Lemma \ref{L:main}.  This involves constructing a
weakly subharmonic
function on each of $S_+$ and $S_-$, strictly subharmonic away from the
dividing curves,
but harmonic near the boundary.
\begin{proof}[Proof of Lemma \ref{L:main}]
  Observe first that $S_+$ and $S_-$ admit Stein
  structures since they are open Riemann surfaces.
  Furthermore, recall that $\partial S_+ = \Gamma = \partial S_-$,
  with opposite orientations.
  We now apply the following Lemma (whose proof we defer) to each of $S_+$
  and $S_-$.

\begin{lemma} \label{L:workhorse}
Let $(\Sigma, j)$ be a compact Riemann surface with boundary.\\
Suppose $f : \Sigma \rightarrow \R$ is bounded below,
$f^{-1}(-1) = \partial \Sigma$ and
$ -d (df \circ j ) = \omega_0$ is a volume form compatible with $j$.

Then, there exist $\epsilon > 0$, a non-positive smooth function
$g : \Sigma \cup [-1, 0] \times \partial \Sigma \rightarrow (-\infinity,
0]$,
an extension of $j$ to $\Sigma \cup [-1, 0] \times \partial \Sigma$,
and a volume form $\omega$ on $\Sigma \cup [-1, 0] \times \partial \Sigma$
compatible with $j$,
with the following properties :
\begin{align*}
	j &= i \text{ on } [-1, 0] \times \Sigma\\
	g &= f \text{ on } \Sigma, \\
	g &< 0 \text{ on } \Sigma \cup [-1, 0) \times \partial \Sigma \\
	g&|_{[-\epsilon, 0] \times \partial \Sigma} : (s,t) \mapsto 2s, \\
	\omega&|_{[-\epsilon, 0] \times \partial \Sigma}  = ds \wedge
	dt \\
	-d ( dg &\circ j ) \ge 0 \text{ on } \Sigma \cup [-1, 0],
	\text{and } &-d ( dg \circ j ) = \omega \text{ on } \Sigma.
\end{align*}
\end{lemma}

  Let $g_\pm$ be the (weakly) subharmonic functions, and let
  $\omega_\pm$ be the area forms from Lemma \ref{L:workhorse}.

Define the following function on
$S = S_- \cup [-1, 1] \times \Gamma \cup S_+$
by :
\begin{equation}
  F = \begin{cases} g_- &\text{on $S_- \cup [-1, 0] \times \partial
  S_-$}\\
	      -g_+(-s, t) &\text{on $[0, 1] \times \partial S_+$} \\
	      -g_+ &\text{on $S_+$}.
		  \end{cases}
\end{equation}

Note that $F$ then defines a smooth function on $S$, since for some
$\epsilon >0$,
$g_-(s,t) = 2s$ for $(s, t) \in [-\epsilon, 0] \times \bdy S_-$ and
$-g_+(-s, t) = 2s$ for
$(s, t) \in [0, \epsilon] \times \bdy S_+$.
Furthermore, the area form defined by :
\begin{equation}
  \omega = \begin{cases}
    w_- &\text{ on $S_- \cup [-1, 0] \times \partial S_-$} \\
    w_+ &\text{ on $[0, 1] \times \bdy S_+ \cup S_+$ }
  \end{cases}
\end{equation}
is a smooth area form on $S$.

It then follows that
\begin{align*}
  d( d F \circ j ) &= -\omega \text{ on $S_-$ } \\
	      &\le 0 \text{ on $S_- \cup [-1, 0] \times \bdy S_-$} \\
	      &=0 \text{ on $[-\epsilon, \epsilon] \times \Gamma$}\\
	      &\ge 0 \text{ on $[0,1]\times \bdy S_+ \cup S_+ $}\\
	      &= \omega \text{ on $S_+$}.
\end{align*}

By scaling $F$, we may set $|F| < \frac{\pi}{2}$.
Furthermore, by construction, $H(s, t) = Cs$ for $s$
close to $0$ in $[-1, 1] \times \Gamma$, and $C > 0$ a constant.
\end{proof}

We now present the proof of Lemma \ref{L:workhorse}.  This uses 
the fact that the Stein structure
on $\Sigma$ may be extended to a cylindrical end glued at the boundary.
We then deform the standard model of the cylindrical end to obtain
the desired weakly subharmonic function to have linear growth at the end.
In essence, this deformation
smoothes a strictly monotone, piecewise-smooth, convex function
on $\R$ to obtain a smooth convex function with a prescribed zero.

\begin{proof}[Proof of Lemma \ref{L:workhorse}]

  Denote by $\partial_i \Sigma$, $i = 1\cdots,N$, be the components of
  the boundary $\partial \Sigma$.
  First, complete $\Sigma$ by gluing the cylinder $[-1, +\infinity)
  \times S^1$
  to each boundary component $\partial_i \Sigma$.  Denote each of these
  cylinders by $Z_i$.
  We then extend the subharmonic function $f$ to $Z_i$ by the
  function given in the cylinder coordinates by
  \[
  f_i(s, t) = A_i (\e^{s} - 1/\e) - 1
  \]
  and
  extend the complex structure to the cylinders by $i$.
  This then extends the symplectic form by
  $\omega_0 = A_i \e^s ds \wedge dt$.
  Denote these extensions again by $f$, $j$, and $\omega_0$,
  which are now defined on $\Sigma \cup \bdy \Sigma \times [-1,
  \infinity)$.
  By scaling $f$ (and thus $\omega_0$)
  as necessary, we may assume $A_i < \frac{1}{2}$.

  At each boundary component, we now apply the following technical lemma,
  whose proof is a simple calculus exercise (see Figure \ref{F:calculus}).
  \begin{lemma} \label{L:calculus}
    For each constant $A \le 1$, there exists a function
    $G_A : [-1, 1] \to \R$ with the following properties:
    \begin{enumerate}
      \item $G_A'(s) > 0$ for all $s \in [-1, 1]$,
      \item $G_A''(s) \ge 0$ for all $s \in [-1, 1]$,
      \item $G_A(s) = A(\e^s -1/\e) -1$ for $s$ near $-1$,
      \item $G_A(s) = 2s$ for $s$ near $0$.
    \end{enumerate}
  \end{lemma}
  \begin{figure} 
    \includegraphics[width=2in]{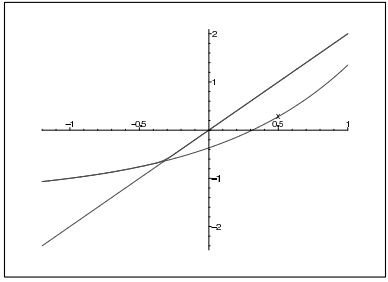}
    \caption{The desired function $G_A$ is obtained by interpolating
    between the two functions $A(\e^s - 1/\e) -1$ and $2s$, preserving
    convexity.}
    \label{F:calculus}
  \end{figure}


For each $i = 1\cdots{N}$, let $G_{A_i}$ be as given
by Lemma \ref{L:calculus}.
Introduce the following function, $g$
defined on $\Sigma \cup [-1, 0] \times \partial \Sigma$ :
\begin{equation*}
  g = \begin{cases} f &\text{on $\Sigma$} \\
    G_{A_i} &\text{ on $[-1, 0] \times \partial_i \Sigma$}
  \end{cases}
\end{equation*}

Then, observe :
\begin{equation} \label{E:subharmonic}
d( d g \circ j ) = \begin{cases} - \omega_0 &\text{ on $\Sigma$ } \\
		      - G_A''(s)ds\wedge dt
		      &\text{ on $[-1, 0] \times \partial_i \Sigma$} .
		    \end{cases}
    \end{equation}
Also note that in any of the $(s, t)$ coordinates near
$\partial_i \Sigma$, $g(s, t) = 2s$ for $s$ sufficiently close to $0$.

Define :
\[
\omega = \begin{cases} \omega_0 &\text{ on $\Sigma$ } \\
			\mu_i(s) ds \wedge dt
			&\text{ on $[-1, 0] \times \partial_i \Sigma$}
	\end{cases}
\]
where $\mu_i(s) > 0$, with the properties that $\mu_i(s) = A_i\e^s$ for
$s$ near $\pm 1$  and $\mu_i(s) = 1$ for $s$ near $0$.
Thus, from Equation \eqref{E:subharmonic}, there exists a non-negative
function $K : \Sigma \cup [-1, 0] \times \bdy \Sigma \rightarrow \R$
so that
\[
d ( dg \circ j ) = - K \omega.
\]
Furthermore, $K = 1$ on $\Sigma$ and $K = 0$ for $s$ near $0$ in $[-1,
\infinity) \times \bdy \Sigma$.
This therefore constructs the desired $g$ and $\omega$.
\end{proof}

\subsection*{Acknowledgements}
I would like to thank Richard Siefring and Joe Coffey for suggesting this
problem to me, and for explaining why it is a natural condition to require.  
I would also like to thank Ko Honda for helpful discussions and encouragement, and
Paolo Ghiggini for both encouragement and helpful comments on a draft.


\end{document}